\documentclass[12pt,reqno]{amsart}
\usepackage{}
\usepackage{a4wide}
\numberwithin{equation}{section}
\usepackage{mathrsfs}
\usepackage{amsfonts}
\usepackage{amsmath,extpfeil}
\usepackage{stmaryrd}
\usepackage{amssymb}
\usepackage{amsthm}
\usepackage{mathrsfs}
\usepackage{url}
\usepackage{amsfonts}
\usepackage{amscd}
\usepackage{indentfirst}
\usepackage{enumerate}
\usepackage{amsmath,amsfonts,amssymb,amsthm}
\usepackage{amsmath,amssymb,amsthm,amscd}
\usepackage{graphicx,mathrsfs}
\usepackage{appendix}
\usepackage[numbers,sort&compress]{natbib}
\usepackage{color}
\usepackage[colorlinks, linkcolor=blue, citecolor=blue]{hyperref}

\makeatletter
\@namedef{subjclassname@2020}{\textup{2020} Mathematics Subject Classification}
\makeatother

\newcommand{\R}{\mathbb{R}}

\renewcommand{\S}{\mathbb S}

\newcommand{\M}{\mathcal M}

\newtheorem{theorem}{Theorem}[section]
\newtheorem{lemma}{Lemma}[section]
\newtheorem{corollary}{Corollary}[section]
\newtheorem{proposition}{Proposition}[section]
\newtheorem{remark}{Remark}[section]    
\numberwithin{equation}{section}

\newcommand\vep{\eps }

\newcommand\eps{\varepsilon}
\newcommand{\p}{\partial}

\newcommand{\Om}{\Omega}
\newcommand{\beq}{\begin{equation}}
\newcommand{\eeq}{\end{equation}}

\begin{document}

\title[$p$-Gauss curvature flow]
{Regularity of \\
the $p$-Gauss curvature flow with flat side}

\author[G. Huang, X.-J. Wang and Y. Zhou]
{Genggeng Huang, Xu-Jia Wang and Yang Zhou}

\address[Genggeng Huang]
{School of Mathematical Sciences, Fudan University, Shanghai 200433, China.}
\email{genggenghuang@fudan.edu.cn}

\address[Xu-Jia Wang]
{Mathematical Sciences Institute, The Australian National University, Canberra, ACT 2601, Australia}
\email{Xu-Jia.Wang@anu.edu.au}

\address[Yang Zhou]
{Mathematical Sciences Institute, The Australian National University, Canberra, ACT 2601, Australia}
\email{yang.zhou18798@outlook.com}

\thanks{The first author was supported by NNSFC 12141105,
the second author was supported by ARC DP200101084 and DP230100499.}


\date{}
\maketitle

\begin{abstract}
We study the regularity of the $p$-Gauss curvature flow with flat side.
In our previous paper \cite{HWZ}, we obtained the regularity of the interface,
namely the boundary of the flat part.
In this paper, we study the regularity of the convex hypersurface near the interface.
\end{abstract}

\baselineskip15.2pt
\parskip5pt

\vspace{3mm}
\section{Introduction}\label{intro}
\vspace{3mm}

Let  $\mathcal M_0$ be a closed convex hypersurface in $\R^{n+1}$, parametrized by
$X_0(\omega)$,  $\omega\in \mathbb{S}^n$.
In this paper we study the Gauss curvature flow with power $p>0$,  
\begin{equation}\label{GCF-p}
\begin{split}
\frac{\partial X}{\partial t}(\omega,t)&=-K^p(\omega,t)\gamma(\omega,t),\\
X(\omega,0)&=X_0(\omega),
\end{split}
\end{equation}
where $K$ is the Gauss  curvature of $\mathcal M_t=X(\omega, t)$,
$\gamma$ is the outer unit normal of $\mathcal M_t$ at $X(\omega, t)$.

The Gauss curvature flow has been extensively studied
if the initial hypersurface $\mathcal M_0$ is strictly convex \cite{A1999, ACGL, BCD2017,C1985, T1985}.
Here we are concerned with the regularity in the case when the initial hypersurface $\M_0$ contains a flat side,
a question first studied by Hamilton \cite{H1993}.
In this case, the solution will become strictly convex instantly when $t>0$ if $p\le \frac 1n$  \cite{A2000},
but the flat side will persist for a while before $\M_t$ becomes strictly convex if $p> \frac 1n$
\cite{H1993, A2000}.
In the latter case,
the local $C^\infty$ regularity of the strictly convex part of $\M_t$ was proved in \cite{T1985, C1985},
and the $C^{1,\alpha}$ regularity across the interface $\Gamma_t$ were obtained in \cite{DS2009},
where $\Gamma_t$ denotes the boundary of the flat side $F_t\subset\M_t$.

When $p=1$, the regularity of $\Gamma_t$ was obtained in \cite{DH1999} for small time $t>0$,
and the long time regularity of $\Gamma_t$ in the case $n=2$ was obtained in \cite{DL2004},
under certain non-degenerate conditions on the initial hypersurface $\mathcal M_0$.
The results in \cite{DL2004} were extended to $p\in (1/2,1]$ in \cite{KLR2013} when $n=2$.
For general $n\ge 2$ and $p>\frac 1n$, the long time regularity of  $\Gamma_t$
was recently obtained by the authors \cite{HWZ}.
We proved that the interface $\Gamma_t$ is smooth until it disappears.

In this paper,
we study the regularity of the strictly convex part of $\M_t$ near the interface $\Gamma_t$.
The regularity of $\mathcal M_t$ near $\Gamma_t$ does not follow directly from the regularity
theory of parabolic equations,  as the flow \eqref{GCF-p} is strongly degenerate near $\Gamma_t$,
even though the regularity of $\Gamma_t$ has been obtained \cite{HWZ}.
For simplicity we assume that $\M_0$ has only one flat part.
Choosing the coordinates properly, we may assume that $\M_t\subset\{y_{n+1}\ge 0\}$ and
the flat side lies on the plane $\{y_{n+1}=0\}$.
Then, locally $\M_t$  can be represented as the graph of a nonnegative function $v$,
\begin{equation*}
y_{n+1}=v(y_1, \cdots, y_n,t)
\end{equation*}
over a bounded domain $\Omega_t$, such that $\Gamma_t$ is strictly contained in $\Omega_t$,
$|Dv|\to\infty$ near $\p\Om_t$, and $v$ satisfies the equation
\begin{equation}\label{v}
v_t(y, t)=\frac{(\det D^2 v(y, t))^p}{(1+|D v|^2)^{\frac {(n+2)p-1}{2}}}\quad y\in\Om_t, \ t>0.
\end{equation}
By the $C^{1,\alpha}$ regularity \cite{DS2009}, we have
$|Dv(y,t)|\to 0$ as $y\to  \Gamma_t$.

For the short time smoothness of the interface $\Gamma_t$,
it is necessary to assume  certain non-degeneracy conditions
on the initial hypersurface $\M_0$ \cite{DH1999,DL2004, D2014}.
Denote
\beq\label{def-g}
g=\big(\text{\Small$\frac{\sigma_p+1}{\sigma_p}$} v\big)^{\frac{\sigma_p}{\sigma_p+1}},\ \ \ \sigma_p=n- \text{\Small$\frac1p$}.
\eeq
The following non-degeneracy conditions were introduced in \cite{DH1999,DL2004, D2014}.

\begin{itemize}
\item[(I1)] The level set $\{v(y, 0) = \eps \}$ is uniformly  convex for $\eps\ge 0$ small,
i.e., its principal curvatures have positive upper and lower bounds.
\item[(I2)] There exists a constant $\lambda_0\in (0, 1)$
such that $\lambda_0\le |Dg(y,0)|\le \lambda_0^{-1} \ \text{on} \ \Gamma_0$.
\end{itemize}

\noindent
Note that condition (I2) implies that  $v(y,0)\approx \text{dist}(y, \Gamma_0)^{(\sigma_p+1)/\sigma_p}$.
We also assume

\begin{itemize}
\item[(I3)]  $\M_0$ is locally uniformly convex and smooth away from the flat region, and
$g(y,0)\in C_\mu^{2+\alpha}(\overline{\{v>0\}})$,  where $C_\mu^{2+\alpha}$ will be introduce in \eqref{1.15} below.
\end{itemize}

We have the following regularity for the function $g$ near the interface $\Gamma_t$.

\begin{theorem}\label{thm-g}
Assume conditions {\rm (I1)-(I3)}.
Then if $p>\frac 1n$, we have $g(\cdot, t)\in C_\mu^{2+\beta}(\overline{\{v>0\}})$ on $0< t <T^*$ for some $\beta\in(0,1)$.
Moreover,
\begin{itemize}
\item[$(1)$] if $\frac{2}{\sigma_p}\in \mathbb Z^+$,
$g$ is  $C^\infty$-smooth up to   $\Gamma_t$ for $0< t <T^*$;

\item[$(2)$] if $\frac{2}{\sigma_p}\notin \mathbb Z^+$,
$g\in C_\mu^{[\frac{2}{\sigma_p}], 2+\beta_0}(\overline{\{g>0\}})$ for $0< t <T^*$,
where $\beta_0 = \min\{1, \frac{4}{\sigma_p} - 2\big[\frac{2}{\sigma_p}\big] \}$.
\end{itemize}
\end{theorem}

We remark that the regularity for $g$ in Theorem \ref{thm-g} is optimal,
due to the term $g(y,t)^{\frac{2}{\sigma_p}} \approx \text{dist}(y,\Gamma_t)^{\frac{2}{\sigma_p}}$
in the equation \eqref{g-eq}.
As usual we use $[a]$ to denote the greatest integer less than $a$.

From Theorem \ref{thm-g}, it follows the regularity of the height function $v$.

\begin{corollary}\label{cor-v}
Assume conditions {\rm (I1)-(I3)} and assume $p>\frac 1n$.
\begin{itemize}
\item[$(1)$] If $\frac{1}{\sigma_p}\in \mathbb Z^+$,
$v$ is  $C^\infty$-smooth up to   $\Gamma_t$ for $0< t <T^*$;

\item[$(2)$] if $\frac{1}{\sigma_p}\notin \mathbb Z^+$,
$v \in C^{1+\big[\frac{1}{\sigma_p}\big], \frac{1}{\sigma_p} - \big[\frac{1}{\sigma_p}\big]}(\overline{\{v>0\}})$ for $0< t <T^*$.
\end{itemize}

Moreover, we have the following intermediate estimate,

\begin{equation}\label{v-inter}
\sup_{t\in[\sigma,T]}\sup_{y,\tilde y\in\overline{\{0<v(\cdot,t)<1\}}} d_{y,\tilde y}(t)^{1+\frac{1}{\sigma_p}} \frac{ |D_y^{k_0+2} v(y,t) - D_{ y}^{k_0+2} v(\tilde y, t) | }{|y-\tilde y|^{\frac{2}{\sigma_p}-k_0}} \le C,
\end{equation}
for all $0<\sigma<T<T^*$,
where $d_{y,\tilde y}(t) := \min \{\text{dist}(y,\Gamma_t), \text{dist}(\tilde y,\Gamma_t)\}$, $k_0$ is the greatest integer strictly less than $\frac{2}{\sigma_p}$ and $C$ is a positive constant depending only on $\mathcal M_0, n, p, \sigma, T$.
\end{corollary}

For intermediate estimate to uniformly elliptic and parabolic equations,
we refer the readers to Chapter IV  \cite{L1996}.

To prove Theorem \ref{thm-g}, we introduce the Hodograph transformation $h$, given in \eqref{gtoh},
which satisfies the evolution equation
\begin{equation}\label{1-h-eq2}
h_t=\frac{(\det \widetilde H)^p}{\big( h_{n+1}^2 + y_{n+1}^{\frac2{\sigma_p}}(1+ h_1^2 +\cdots + h_{n-1}^2) \big) ^{\frac {(n+2)p-1}{2}}}  \ \  \text{in} \ \   \{y_{n+1}>0\},
\end{equation}
where the matrix $\widetilde H$ is given in \eqref{tH-1}.
Note that equation \eqref{1-h-eq2} is a degenerate fully nonlinear parabolic equation
without the concavity condition, and the coefficient $y_{n+1}^{\frac2{\sigma_p}}$ is only H\"older continuous when $p>\frac {1}{n-2}$.
Hence the regularity theories, such as  \cite{JW2014,L1996}, do not apply.
Here we make use of some estimates in \cite{HTW}.

The paper is organized as follows.
In Section \ref{s2}, we recall the short time regularity and some basic estimates.
We then derive the equations for $g,h$ and prove the regularity (Theorem \ref{thm-g}) in Section \ref{s7}.

\vspace{2mm}

\noindent{\bf Notation.}
Given two positive quantities $a$ and $b$, we denote
$ a\lesssim b$
if there is a constant $C>0$, depending only on $\mathcal M_0, n, p, T$,
such that $a \le C b$, where $T\in (0, T^*)$ is any given constant.
We also denote
$a\approx b$
if  $a\lesssim b$ and $b\lesssim a$.

Let $k\ge 0$ be an integer and  $\alpha\in (0, 1]$.
Let $\Omega$ be a domain in $\mathbb R^n$.
As usual, we define the norm $\|\cdot\|_{C^{k,\alpha}(\overline \Omega)}$ by
\beq \label{norm1}
\| U\|_{C^{k,\alpha}(\overline \Omega)}=\sup_{|\gamma|\le k}|D^\gamma U(x)|
  + \sup_{|\gamma| = k \atop x,y\in\Omega}\frac{|D^\gamma U(x)-D^\gamma U(y)|}{|x-y|^\alpha}.
\eeq
In the parabolic case, we denote
\beq \label{norm2}
\| U\|_{C^{k+\alpha, \frac{k+\alpha}{2}}_{x,t}(\overline Q)}
  =\sup_{|\gamma|+2s\le k \atop (x,t)\in Q}|D_x^\gamma D^s_t U(x,t)| + \sup_{|\gamma|+2s
  = k \atop (x,t), (y,t')\in Q}\frac{|D_x^\gamma D^s_t U(x,t)-D_x^\gamma D^s_t U(y,t')|}{(|x-y|^2 +|t-t'|)^{\alpha/2}},
\eeq
where $Q$ is a domain in $\R^n\times\R^1$.
If $\alpha\in(0,1)$, we will write $\|\cdot \|_{C^{k+\alpha, \frac{k+\alpha}{2}}_{x,t}(\overline Q)}$ as $\|\cdot \|_{C^{k+\alpha}(\overline Q)}$ for brevity.

To study the regularity of $g$,
we introduce H\"older spaces
with respect to the metric $\mu$ in $ \R^{n-1}\times \mathbb R^+ \times \R$ as in \cite{DH1999, DL2004},
\begin{equation*}
\mu[(x, t), (y, s)] = |x'-y'| + |\sqrt{x_{n}} - \sqrt{y_{n}}| + \sqrt{|t-s|}.
\end{equation*}
Let $Q$  be a domain in $\mathbb R^{n,+} \times \R$,  where $\R^{n, +} :=\R^{n-1}\times \R^+=\{x\in\R^n\ |\  x_n>0\}$.
We denote
\begin{equation}
\|U\|_{C_\mu^{0,\alpha}({\overline Q})}
                = \sup_{p\in {Q}} |U(p) |+  \sup_{p_1,p_2\in{Q}} \frac{|U(p_1)-U(p_2)|}{\mu[p_1,p_2]^\alpha},
\end{equation}
\begin{equation}\label{1.15}
\begin{split}
\|U\|_{C_\mu^{2+\alpha}({\overline Q})}
  & = \|x_n U_{nn}\|_{C_\mu^{0,\alpha}({\overline Q})}
         + {\Small\text{$\sum_{i=1}^{n-1}$}} \|\sqrt{x_n}U_{ni}\|_{C_\mu^{0,\alpha}({\overline Q})}
        + {\Small\text{$\sum_{i,j=1}^{n-1}$}} \|U_{ij}\|_{C_\mu^{0,\alpha}({\overline Q})} \\
& \quad   + {\Small\text{$\sum_{i=1}^{n}$}} \|U_{i}\|_{C_\mu^{0,\alpha}({\overline Q})}
               +  \|U_{t}\|_{C_\mu^{0,\alpha}({\overline Q})}+ \|U\|_{C_\mu^{0,\alpha}({\overline Q})} ,
\end{split}
\end{equation}
and
\begin{equation}\label{1.16}
\|U\|_{C_\mu^{m,2+\alpha}({\overline Q})}
            =\sum_{|\gamma|+2s\le m}\|D^\gamma_xD^s_t U\|_{C_\mu^{2+\alpha}({\overline Q})}.
\end{equation}

\vspace{3mm}
\section{Some estimates}\label{s2}
\vspace{3mm}

First we recall the short time existence and regularity in \cite{DH1999},
where Daskalopoulos and Hamilton proved the following.

\begin{proposition}[Theorem 9.1,\cite{DH1999}]\label{DH1999}
Assume the conditions {\rm (I1)-(I3)}.
Then, there exists a time $T_0>0$
such that \eqref{GCF-p} admits a solution $\mathcal M_t$ for $0< t\leq T_0$, and
at any given time $t\in (0, T_0]$, $\M_{t}$ satisfies the conditions {\rm (I1)-(I3)}.
\end{proposition}
\begin{remark}\label{rm1}
\rm{Proposition \ref{DH1999} is proved in} \rm{\cite{DH1999}} for $n=2$. The proof also holds for high dimension case $n\ge 3$. Moreover, for $0<t\le T_0$,  the proof also implies the following condition (see Theorem 9.2 in \cite{DH1999}):
\begin{itemize}
\item[(I4)]
 $g_{ij}\tau_i g_j\in L^\infty(\{v>0\})$,
 where $\tau=(\tau_1,\cdots,\tau_n)$ is any tangent vector field of the level set of $g$, i.e.  $\tau\cdot \nabla g=0$.
\end{itemize}
Therefore, choosing a sufficiently small $t_0>0$ as the initial time, we may assume that (I4) holds at $t=0.$
\end{remark}

To prove Theorem \ref{thm-g}, we then tap into some estimates obtained in previous work \cite{HWZ}.

Let $u(\cdot, t)$ be the Legendre transformation of $v(\cdot, t)$, i.e.
\begin{equation}\label{def-u}
u(x,t)=\sup\{y\cdot x-v(y,t)\ |\ y\in\Omega_t\}, \ \   x\in D_yv(\Omega_t)=\R^n.
\end{equation}
Then $u(x,t)$ satisfies the equation
\begin{equation}\label{u}
\det D^2 u=\frac{1}{(-u_t)^{\frac 1p}(1+|x|^2)^{\frac{(n+2)p-1}{2p}}}+c_t\delta_0,
\end{equation}
where $c_t$ is the volume of the flat part.  Hence $c_t>0$ for $t\in [0,T^*)$.
Without loss of generality,  we assume that the origin is an interior point of the convex set $\{v(\cdot, t)=0\}$ for all $t\in[0,T^*)$.
Then for any given $T\in (0, T^*)$, there is a positive constant $\rho_0$ such that
\begin{equation}\label{rho-0}
B_{\rho_0}(0)\subset\subset  \{y\in\mathbb R^n \ |\  v(y,t)=0\}, \ \   \forall~t\in[0,T].
\end{equation}

\begin{lemma}[\cite{HWZ}] \label{lemma-u}
Assume the conditions {\rm (I1)-(I4)}.
Then
\beq\label{urr2}
{\begin{split}
-u_{t}(x,t) & \approx  |x|, \\
u_{rr}(x,t) & \approx |x|^{n-1-1/p}, \\
u_{\xi\xi}(x, t) & \approx |x|^{-1},
\end{split}} \eeq
for any $x\in B_1(0)\backslash \{0\}$, $t\in [0,T]$ and
any unit vector $\xi\perp \overrightarrow{ox}$, where
$u_{rr}(x,t) := \frac1{|x|^2} x_ix_j u_{ij}(x,t)$.
\end{lemma}

Denote $r=|x|$.
Let
\begin{equation}\label{Trans2}
   \zeta(\theta, s, t) =\frac{u(\theta, r, t)}{r},\ \ \
   s =r^{\frac{\sigma_p}2},
\end{equation}
where $(\theta, r)$ is the spherical coordinates for $x$.
Then $\zeta$ satisfies the  parabolic Monge-Amp\`ere type equation \cite{HWZ}:
\begin{equation}\label{1-po2}
-\zeta_t\det \begin{pmatrix}
 \zeta_{ss}+\frac{2+\sigma_p}{\sigma_p}\frac{\zeta_s}{s}   &   \zeta_{s\theta_{1}}&\cdots& \zeta_{s\theta_{n-1}}\\[3pt]
 \zeta_{s\theta_1} & \zeta_{\theta_1\theta_1}+\zeta+\frac {\sigma_p}2 s\zeta_s & \cdots&\zeta_{\theta_1\theta_{n-1}}\\[3pt]
                    \cdots &\cdots &\cdots &\cdots  \\[3pt]
 \zeta_{s\theta_{n-1}} &\zeta_{\theta_1\theta_ {n-1}} & \cdots&\zeta_{\theta_{n-1}\theta_{n-1}}+\zeta+\frac {\sigma_p}2 s\zeta_s
\end{pmatrix}^p=\bar {F} (s),
\end{equation}
in $\{s>0\}$,
where
$\bar {F} (s)= 4^p \sigma_p^{-2p} \big(1+s^{4/\sigma_p}\big)^{-\frac{(n+2)p-1}2} . $

\begin{lemma}[Theorem 6.3, \cite{HWZ}]\label{lemma-zeta} \
Assume the conditions {\rm (I1)-(I4)}. We have
\begin{equation}
\|\zeta\|_{ C^{2+\alpha_0}(\mathbb S^{n-1}\times [0,1]\times [\sigma,T])}\le C, \ \ \ \forall\ 0<\sigma<T<T^*,
\end{equation}
where the constants $\alpha_0\in(0,1)$ and $C>0$ depend only on $\mathcal M_0, n, p, \sigma, T$.
\end{lemma}

Next we quote the $C^{\alpha}$ and $C^{2,\alpha}$ estimates
for degenerate linear parabolic equations which are needed later.
Given a point $p_0=(x_0, t_0)=(x'_0,x_{0,n},t_0)\in \mathbb R^{n,+} \times \R$ , denote
\begin{equation}\label{cylinder}
Q^*_\rho(p_0)=\{(x, t)\ |\
 x_n > 0, |x'-x'_0|< \rho, |x_n-x_{0,n}| < \rho^2, t_0-\rho^2 < t\le t_0\} ,
\end{equation}
which is a cylinder in $\mathbb R^{n,+} \times \R$.
When $p_0=(0, 0)$, we simply write  $Q^*_\rho=Q^*_\rho(p_0)$.

Consider the following linear degenerate operator
\begin{equation}\label{l++}
L_+ U:= - U_t + a_{nn}x_n\p_{nn} U
    + {\Small\text{$\sum_{i=1}^{n-1}$}} 2a_{in} \sqrt{x_n} \p_{in} U
    + {\Small\text{$\sum_{i,j=1}^{n-1}$}} a_{ij} \p_{ij} U
    +{\Small\text{$\sum_{i=1}^n$}}b_i \partial_i U      
\end{equation}
with variable coefficients $a_{ij}, b_i$ defined in the cylinder $Q^*_\rho$.

\begin{lemma}\label{Lemma-2.3}
Assume that the coefficients $a_{ij}, b_i$ are measurable
and satisfy
\begin{equation*}
\begin{split}
 &  a_{ij}\xi_i\xi_j  \ge \lambda  |\xi|^2  \ \ \  \forall\ \xi\in\mathbb R^n, \\
 & |a_{ij} |,  |b_{i}| \le \lambda^{-1},
  \end{split}
\end{equation*}
and
\begin{equation*}
\frac{2b_n}{a_{nn}}\ge \nu, \hskip50pt
\end{equation*}
for some constants $\lambda,\nu\in(0,1)$.
Let $U\in C^2(\overline {Q_\rho})$ be the solution to $L_+ U = f$.
Then there exists $\alpha\in (0, 1)$ such that  for any $\rho' \in(0, \rho)$, it holds 
\begin{equation}\label{a418z}
\|U\|_{ C_{\mu}^\alpha(Q^*_{\rho'})}
     \le C \Big(\sup_{Q^*_\rho}|U|+\Big(\int_{Q^*_\rho} |f|^{n+1} x_n^{\frac{\nu}{2}-1} dx dt\Big)^{\frac1{n+1}}\Big),
\end{equation}
where the positive constant $C$ depends only on $n, \rho, \rho', \lambda$ and $\nu$.
\end{lemma}

For the proof of Lemma \ref{Lemma-2.3},
we refer the readers to \cite[Theorem 3.1]{DL2003} or \cite[Theorem 3.3]{L2016}.

\begin{lemma}[Schauder estimate \cite{DH1999}]\label{Lemma-2.4}\
Assume that the coefficients $a_{ij}, b_i\in C_\mu^{\alpha}(\overline{Q^*_\rho})$
for some $\alpha\in(0,1)$ and satisfy
\begin{equation*}
\begin{split}
  & a_{ij}\xi_i\xi_j\ge \lambda  |\xi|^2  \ \ \  \forall\ \xi\in\mathbb R^n, \\
  &  \|a_{ij}\|_{C_\mu^{\alpha}(\overline{Q^*_\rho})}, \  \ \|b_{i}\|_{C_\mu^{\alpha}(\overline{Q^*_\rho})} \le \lambda^{-1},
  \end{split}
\end{equation*}
and
\begin{equation*}
b_n\ge \lambda  \ \   \text{at} ~~\, \{x_n=0\}
\end{equation*}
for some constant $\lambda\in(0,1)$. 
Let $U\in C_\mu^{2+\alpha}(\overline{Q^*_\rho})$ be the solution to $L_+ U = f$.
Then for any given $\rho'\in(0,\rho)$, it holds
\begin{equation}
 \|U\|_{C_\mu^{2+\alpha}(\overline{Q^*_{\rho'}})}
       \le C \Big(\|U\|_{L^\infty(\overline{Q^*_\rho})} + \|f\|_{C_\mu^{\alpha}(\overline{Q^*_\rho})}\Big),
\end{equation}
where the positive constant $C$ depends only on $n, \alpha, \rho, \rho'$ and $\lambda$.
\end{lemma}

\vspace{3mm}
\section{The regularity for the graph}\label{s7}
\vspace{3mm}

In this section, we will first derive the evolution equations of $g$ and $h$. Then, we utilize the a priori estimates of $u$ and $\zeta$ to obtain the $C^{2+\beta}_\mu$ regularity of $g$ and $h$,  which allows us to prove Theorem \ref{thm-g} and Corollary \ref{cor-v}.

\subsection{\bf Derivation of equations.}
Recall that the function $v$ satisfies equation \eqref{v} and $g$ is defined in \eqref{def-g}.
A direct computation yields that,
for $1\le i,j \le n$,
\begin{equation}\label{v-g-1}
v_i= g^{\frac{1}{\sigma_p}} g_i,  \ \   v_t= g^{\frac{1}{\sigma_p}} g_t,
\end{equation}
and
\begin{equation}\label{v-g-2}
v_{ij}= g^{\frac{1}{\sigma_p}} g_{ij} + \frac{1}{\sigma_p}g^{\frac{1}{\sigma_p}-1} g_{i}g_j.
\end{equation}
Then, $g$ satisfies
\begin{equation}\label{g-eq}
g_t=\frac{\left(g \det \big(D^2 g + \frac1{\sigma_p}g^{-1} Dg\otimes Dg\big)\right)^p}{\big(1+g^{\frac2{\sigma_p}}|D g|^2\big)^{\frac {(n+2)p-1}{2}}}.
\end{equation}
Here $Dg\otimes Dg$ is a matrix with $(i,j)$-entries $g_ig_j$.

Then, we perform the following Hodograph transformation mentioned in Section \ref{intro}.
Let $\bar p_0 = (\bar y_0, \bar t_0)$ be a point on the interface $\Gamma_{\bar t_0}$ with $0 < \bar t_0 < T^*$.
By a rotation of the coordinates,
we may assume that $e_n=(0, \cdots, 0, 1)$ is the unit outer normal of the flat part $\{v(y,\bar t_0) = 0\}$ at $\bar p_0$, so that at the point, we have
$$
g_i(\bar p_0)=0,  \ \   i=1,\cdots,n-1,  \ \  \text{and}  \ \   g_n(\bar p_0)>0.
$$
The above condition can be guaranteed by the initial conditions on $g(y,0)$ and later by the a priori estimates on $|Dg|$ (see \eqref{gtgx-est} below).  Hence, we can solve the equation $y_{n+1}=g(y_1,\cdots,y_{n},t)$ with respect to $y_n$ around the point $\bar p_0$ and yield a map
\begin{equation}\label{gtoh}
y_n=-h(y_1,\cdots,y_{n-1},y_{n+1},t),
\end{equation}
defined for all $(y_1,\cdots,y_{n-1},y_{n+1},t)$ sufficiently close to $\bar q_0=(\bar y_{0,1},\cdots,\bar y_{0,n-1},0,\bar t_0)$.
Then, direct computations give that
\begin{equation}\label{gtoh-1}
Dg= -\frac{1}{h_{n+1}}(h_1,\cdots,h_{n-1}, 1),  \ \   g_t= - \frac{h_t}{h_{n+1}}
\end{equation}
and for $1\le i,j \le n-1$,
$$
g_{ij}= -\frac{1}{h_{n+1}}\big(h_{ij}- h_{i,n+1}\frac{h_j}{h_{n+1}} -h_{j,n+1}\frac{h_i}{h_{n+1}} +h_{n+1,n+1}\frac{h_ih_j}{h^2_{n+1}}\big),
$$
$$
g_{in}= \frac{1}{h^2_{n+1}}\big(h_{i,n+1} - h_{n+1,n+1} \frac{h_i}{h_{n+1}}\big)
$$
and
$$
g_{nn}= -\frac{1}{h^3_{n+1}}h_{n+1,n+1}.
$$
According to \eqref{g-eq}, $h$ satisfies
\begin{equation}\label{h-eq1}
h_t=\frac{\left(y_{n+1} \det H\right)^p}{\big(h_{n+1}^2 + y_{n+1}^{\frac2{\sigma_p}}(1+ h_1^2 +\cdots + h_{n-1}^2)\big)^{\frac {(n+2)p-1}{2}}},
\end{equation}
where $H$ is an $n\times n$ matrix with entries
$$H_{ij}=h_{ij}- h_{i,n+1}\frac{h_j}{h_{n+1}} - h_{j,n+1}\frac{h_i}{h_{n+1}} +h_{n+1,n+1}\frac{h_ih_j}{h^2_{n+1}} - \frac{\sigma_p^{-1}
h_ih_j}{ y_{n+1}h_{n+1}},$$
$$H_{i,n+1}=h_{i,n+1} - h_{n+1,n+1} \frac{h_i}{h_{n+1}} + \frac{\sigma_p^{-1}h_i}{y_{n+1}}$$
and
$$H_{n+1,n+1}=h_{n+1,n+1} - \frac{\sigma_p^{-1}h_{n+1}}{y_{n+1}}. $$
By the elementary properties of determinants, equation \eqref{h-eq1} can be transformed into
\begin{equation}\label{h-eq2}
h_t=\frac{\left(\det \widetilde H\right)^p}{\big(h_{n+1}^2 + y_{n+1}^{\frac2{\sigma_p}}(1+ |D_{y'}h|^2)\big)^{\frac {(n+2)p-1}{2}}}
\end{equation}
where
\begin{equation}\label{tH-1}
\widetilde H=\begin{pmatrix}
 h_{11} & \cdots& h_{1,n-1}&  \sqrt{y_{n+1}}h_{1,n+1}\\[3pt]
   \cdots &\cdots &\cdots &\cdots \\[3pt]
h_{1,n-1} & \cdots& h_{n-1,n-1}&\sqrt{y_{n+1}} h_{n-1,n+1} \\[3pt]
  \sqrt{y_{n+1}} h_{1,n+1}  & \cdots &   \sqrt{y_{n+1}} h_{n-1,n+1} & y_{n+1} h_{n+1,n+1} - \sigma_p^{-1} h_{n+1}
 \end{pmatrix}.
\end{equation}
One can calculate that the linearized operator of \eqref{h-eq2},
\begin{equation}\label{lin-7}
\begin{split}
\mathcal L := &-\frac{1}{h_t}\partial_t+ {\Small\text{$\sum_{i,j=1}^{n-1}$}} p \widetilde H^{ij}\partial_{y_iy_j}+2 {\Small\text{$\sum_{i=1}^{n-1}$}} p\widetilde H^{i,n+1}\sqrt{y_{n+1}}\partial_{y_iy_{n+1}}+py_{n+1}\widetilde H^{n+1,n+1} \partial_{y_{n+1}y_{n+1}}\\
&+ \hat b \partial_{y_{n+1}}- {\Small\text{$\sum_{i=1}^{n-1}$}} \frac{[(n+2)p-1] y_{n+1}^{\frac{2}{\sigma_p}}}{h_{n+1}^2+ y_{n+1}^{\frac{2}{\sigma_p}}(1+|D_{y'}h|^2)}~ h_i \partial_{y_i},
\end{split}
\end{equation}
where $\widetilde H^{ij}, \widetilde H^{i,n+1},\widetilde H^{n+1,n+1}$,
$i, j=1,\cdots,n-1$ are the elements of the inverse matrix of $\widetilde H$ and
\begin{equation}\label{hat-b}
\hat b (y',y_{n+1},t) := -\Big(\frac{[(n+2)p-1]h_{n+1}}{h_{n+1}^2+ y_{n+1}^{\frac{2}{\sigma_p}}(1+|D_{y'}h|^2)}+\frac{p}{\sigma_p}\widetilde H^{n+1,n+1}\Big).
\end{equation}

\subsection{\bf Regularity for $g$ and $h$}

\begin{lemma}\label{Lemma-3.1}
Assume the conditions (I1)-(I4). Then for some $\beta\in(0,1)$, it holds that
\begin{equation}\label{g-c2}
\max_{t\in[\sigma,T]}\|g(\cdot,t)\|_{C_\mu^{2+\beta}(\overline{\{0<g(\cdot,t) <1\}})} \le C(\mathcal M_0,n,p,\sigma,T),
\end{equation}
for all $0<\sigma<T < T^*$.
\end{lemma}

\begin{proof}
{\bf Step 1: First order derivative estimates}

By \eqref{def-u} and \eqref{urr2}, we have
\begin{equation*}
{\begin{split}
v &=x\cdot D u-u=ru_r -u \\
  &= \int_0^r (\rho u_\rho -u)_\rho d\rho
      =\int_0^r \rho u_{\rho\rho}d\rho\approx r^{\sigma_p+1} .
   \end{split}}
\end{equation*}
Hence $g \approx r^{\sigma_p}$.
It follows from  \eqref{v-g-1} and Lemma \ref{lemma-u} that
\begin{equation}\label{gtgx-est}
g_t\approx 1, \ \   |D g|\approx 1,
\end{equation}
uniformly near the interface $\Gamma_t$ for $t\in(0,T]$.
Note that estimate \eqref{gtgx-est} allows us to perform the local coordinate transform \eqref{gtoh}.

Fix a point $\bar p_0=(\bar y_0, \bar t_0)$ on the interface $\Gamma_{\bar t_0}$, where $\bar t_0 \in (0,T]$.
By a rotation of the coordinates,
we may assume that the unit outer normal of the flat part $\{v(y,\bar t_0) = 0\}$ at $\bar p_0$ is $e_n=(0, \cdots, 0, 1)$, i.e., $\frac{Dg(\bar p_0)}{|Dg(\bar p_0)|}=e_n$. 
By Lemmas 4.7 and 4.8 in \cite{DL2004}, there exist positive constants $\eta> 0$ and $\gamma_0> 0$, depending only on the initial data and $\rho_0$, such that
\begin{equation}\label{kgg-1}
e_n \cdot \frac{Dg( p)}{|Dg(p)|} \ge \gamma_0, \quad \forall~ p=(y,t) \ \ \text{with} \ \ g(p)>0 \ \ \text{and}  \ \ |p-\bar p_0| < \eta, t\in(0,\bar t_0].
\end{equation}

From \eqref{gtoh-1}, \eqref{gtgx-est} and \eqref{kgg-1}, one knows, for small constant $\eta>0$,
\begin{equation}\label{h-dydt}
h_t \approx 1, \quad -h_{n+1} \approx \sqrt{1+ |D_{y'}h|^2} \approx 1, \quad \forall~(y',y_{n+1},t)\in Q^*_\eta(\bar q_0),
\end{equation}
where the cylinder $Q^*_\eta$ is defined in \eqref{cylinder}.

{\bf Step 2: Second order derivative estimates}

Now, we fix a point $(y_0,t_0) \in \{(y, t)~|~  g(y,t)>0,  |(y,t)-(\bar y_0, \bar t_0)| < \eta, \bar t_0 - \eta^2 <t\le \bar t_0\}$ for small constant $\eta>0$.
Let ${\xi}^{(1)},\cdots,{\xi}^{(n-1)}, {\xi^{(n)}}=\frac{D g}{|D g|}$ be $n$ vectors at the point $(y_0,t_0)$ with ${\xi}^{(i)}=\frac{g_ne_i-g_ie_n}{|g_ne_i-g_ie_n|}$, $i=1,\cdots,n-1$. Note that ${\xi}^{(i)} \bot {\xi^{(n)}}$ for $i=1,\cdots,n-1$.
From \eqref{kgg-1}, one gets
\begin{equation}\label{0425-1}
\det ({\xi}^{(1)},\cdots,{\xi}^{(n-1)},{\xi^{(n)}})
\ge \Big(\frac{g_n}{|Dg|}\Big)^{n-2} \approx 1. 
\end{equation}

At point $(y_0,t_0)$, consider the following matrix
\begin{equation}\label{G1}
G=\begin{pmatrix}
   g_{{\xi}^{(1)}{\xi}^{(1)}}&   \cdots  &    g_{{\xi}^{(1)}{\xi}^{(n-1)}} & \sqrt{g}   g_{{\xi}^{(1)}{\xi^{(n)}}} \\[3pt]
   g_{{\xi}^{(2)}{\xi}^{(1)}} & \cdots&g_{{\xi}^{(2)}{\xi}^{(n-1)}} & \sqrt{g} g_{{\xi}^{(2)}{\xi^{(n)}}}\\[3pt]
   \cdots &\cdots &\cdots &\cdots \\[3pt]
 \sqrt{g}   g_{{\xi}^{(n)}{\xi^{(1)}}} &\cdots & \sqrt{g}   g_{{\xi}^{(n)}{\xi^{(n-1)}}}  & g g_{{\xi^{(n)}}{\xi^{(n)}}}+\frac{1}{\sigma_p}|Dg|^2
 \end{pmatrix}.
\end{equation}
Then, \eqref{0425-1} gives that
\begin{equation}\label{kgg-3}
\begin{split}
\det G & =g \det \big(D^2 g + \frac1{\sigma_p}g^{-1} Dg\otimes Dg\big)(\det ({\xi}^{(1)},\cdots,{\xi}^{(n-1)},{\xi^{(n)}}))^2 \\
& \approx g \det \big(D^2 g + \frac1{\sigma_p}g^{-1} Dg\otimes Dg\big).
\end{split}
\end{equation}

The first aim in this part is to show that
\begin{equation}\label{GI}
G\approx I_{n\times n},
\end{equation}
where $I_{n\times n}$ is the identity matrix.
For this,
let  ${\nu}$ be the unit eigenvector which corresponds to the smallest eigenvalue of $D^2 u$.
Recall that $x_0=D_y v(y_0,t_0)$.  
Then $r =|x_0|$ can be arbitrary small if we take $\eta>0$ small.
\begin{equation*}
\text{Claim:} \ \   \ \   \ \   |{\xi^{(n)}}\cdot {\nu}-1| \lesssim r^{\sigma_p} , \ \   |{\xi}^{(i)}\cdot {\nu}|\lesssim r^{\frac{\sigma_p}{2}}, \ \   i=1,\cdots, n-1.  \ \   \ \   \ \   \ \   \ \  \,
\end{equation*}
By the estimate \eqref{urr2} and
$$\xi^{(n)}= \frac{Dg(y_0,t_0)}{|Dg(y_0,t_0)|} = \frac{Dv(y_0,t_0)}{|Dv(y_0,t_0)|} = \frac{x_0}{|x_0|} = \frac{x_0}{r},$$
we have  $u_{{\xi^{(n)}}{\xi^{(n)}}}=u_{rr} \approx r^{\sigma_p-1}$ 
at the point $(x_0,t_0)$ \cite{HWZ}.
Suppose the first part of the Claim fails, i.e. $|{\xi^{(n)}}\cdot {\nu}-1| >> r^{\sigma_p} $. Denote
\begin{equation*}
\begin{split}
{\xi^{(n)}}=\tau_1 {\nu}+ \tau_2{\bar{\xi}}, \quad \text{for~some~unit~vector~}   {\bar{\xi}} \perp {\nu},
\end{split}
\end{equation*}
where $\tau_1={\xi^{(n)}}\cdot \nu$ and $\tau_2={\xi^{(n)}}\cdot \bar\xi$.
Then, it holds
\begin{equation*}
\begin{split}
 \big||\tau_1|-1\big|>> r^{\sigma_p} , \ \   \tau_2 >> r^{\frac{\sigma_p}{2}} ,
\end{split}
\end{equation*}
which yields
\begin{equation*}
\begin{split}
r^{\sigma_p-1} \approx u_{{\xi^{(n)}}{\xi^{(n)}}}&=\tau_1^2 u_{{\nu}{\nu}}+2\tau_1\tau_2 u_{{\nu}{\bar{\xi}}}+\tau_2^2 u_{{\bar{\xi}}{\bar{\xi}}}\\
&\ge (\tau_2\sqrt {u_{{\bar{\xi}}{\bar{\xi}}}}-|\tau_1| \sqrt {u_{{\nu}{\nu}} })^2\\
& >> r^{\sigma_p-1}.
\end{split}
\end{equation*}
 This contradiction proves the first part of the Claim.

 It then follows that
 \begin{equation*}
 |{\xi^{(n)}}-{\nu}|^2=2|{\xi^{(n)}}\cdot {\nu}-1|\lesssim r^{\sigma_p}.
 \end{equation*}
 Hence
 \begin{equation*}
 |{\xi}^{(i)}\cdot {\nu}|=|{\xi}^{(i)}\cdot({\xi^{(n)}}- {\nu})|\lesssim r^{\frac{\sigma_p}{2}}, \quad  i=1,\cdots, n-1,
 \end{equation*}
which proves the second part of the Claim.

By the above Claim and Lemma \ref{lemma-u}, for $\bar{\xi}\perp \nu$, we get
\begin{equation*}
\begin{split}
g_{{\xi}^{(i)}{\xi}^{(i)}}& = \left(\frac{1+\sigma_p}{\sigma_p}v\right)^{-\frac{1}{1+\sigma_p}} v_{{\xi}^{(i)}{\xi}^{(i)}}
 \approx  r^{-1}  u^{{\xi}^{(i)}{\xi}^{(i)}}\\
 &= r^{-1}\Big(u^{{\bar{\xi}}{\bar{\xi}}}( {\bar{\xi}}\cdot {\xi}^{(i)})^2+2u^{{\bar{\xi}}{ \nu}}({\bar{\xi}}\cdot {\xi}^{(i)})({\nu}\cdot {\xi}^{(i)})+u^{{\nu}{\nu}} ({\nu}\cdot {\xi}^{(i)})^2 \Big) \\
& \lesssim  r^{-1}\left(r+  r^{1-\frac{\sigma_p}{2}} \cdot r^{\frac{\sigma_p}{2}}+r^{1-\sigma_p}r^{\sigma_p}\right) \lesssim 1,
\end{split}
\end{equation*}
for $i=1,\cdots, n-1$, and
\begin{equation*}
\begin{split}
 g g_{{\xi^{(n)}}{\xi^{(n)}}}+\frac{1}{\sigma_p}|Dg|^2
   & =g \left(\frac{1+\sigma_p}{\sigma_p}v\right)^{-\frac{1}{1+\sigma_p}}v_{{\xi^{(n)}}{\xi^{(n)}}}\\
  & \approx  r^{\sigma_p-1}u^{{\xi^{(n)}}{\xi^{(n)}}}\\
  & \lesssim r^{\sigma_p-1} r^{1-\sigma_p}\approx 1.
\end{split}
\end{equation*}
Then from equation \eqref{g-eq} and estimates \eqref{gtgx-est}, \eqref{kgg-3}, one knows $G\approx I_{n\times n}$.

We next claim that the matrix $\widetilde H$,  defined in \eqref{tH-1},  satisfies
\begin{equation}\label{HI}
\widetilde H\approx I_{n\times n} \quad\text{in}\  Q^*_\eta(\bar q_0) ,
\end{equation}
where $\bar q_0=(\bar y_0',0,\bar t_0)$.
Indeed, by the definition of $h$ , one has
$$
{\begin{split}
{\xi}^{(i)} & =\frac{e_i-h_ie_n}{\sqrt{1+h_i^2}}, \ \ i=1,\cdots,n-1, \\
{\xi^{(n)}} & =\frac{(D_{y'}h,1)}{\sqrt{1+|D_{y'}h|^2}}\\
\end{split}}
$$
at the point $(y_0', g(y_0,t_0),t_0)$.
Then a direct computation implies
\begin{equation}\label{g-2tn}
\begin{split}
& g_{{\xi}^{(i)}{\xi}^{(j)}}= -\frac{h_{ij}}{h_{n+1}\sqrt{(1+h_i^2)(1+h_j^2)}}, \ \   1\le i,j\le n-1,\\
& g_{{\xi}^{(i)}{\xi^{(n)}}}=\frac{1}{\sqrt{(1+h_i^2)(1+|D_{y'}h|^2)}} \left(-\frac{h_kh_{ik}}{h_{n+1}}+\frac{1+|D_{y'}h|^2}{h_{n+1}^2}h_{i,n+1}\right),\	 1\le i\le n-1,\\
& g_{\xi^{(n)} \xi^{(n)}}= -\frac{h_kh_lh_{kl}}{h_{n+1}(1+|D_{y'}h|^2)} + \frac{2h_l h_{l,n+1}}{h_{n+1}^2}-\frac{(1+|D_{y'}h|^2) h_{n+1,n+1}}{h_{n+1}^3}.
\end{split}
\end{equation}
Here the subscripts $k,l$ obey the Einstein summation convention from $1$ to $n-1$.
From \eqref{h-dydt}, \eqref{GI} and \eqref{g-2tn}  it follows that
\begin{equation}
\sum_{i,j=1}^{n-1}(|h_{ij}|+|\sqrt{y_{n+1}}h_{i,n+1}|)+|y_{n+1}h_{n+1,n+1}|\lesssim 1,
\end{equation}
which gives \eqref{HI} by \eqref{h-eq2} and the arbitrariness of $(y_0,t_0)$.

{\bf Step 3: $C^{2+\beta}_{\mu}$-estimate}

Now we refine the estimates of $g g_{{\xi^{(n)}}{\xi^{(n)}}}$ and $\sqrt{g}   g_{{\xi}^{(i)}{\xi^{(n)}}}$
according to the regularity of $\zeta$.
By Lemma \ref{lemma-zeta},
$\zeta(\theta,s,t)\in C^{2+\alpha_0} (\S^{n-1}\times[0,1] \times(0,T])$.
At the point $(x_0,t_0) = (D_yv(y_0,t_0),t_0)$,
where $(y_0,t_0)$ is a fixed point in $\{(y, t)~|~  g(y,t)>0,  |(y,t)-(\bar y_0, \bar t_0)| < \eta, \bar t_0 - \eta^2 <t\le \bar t_0\}$ for small constant $\eta>0$, one knows that
\begin{equation}\label{u-nn}
{\begin{split}
   u_{{\xi^{(n)}}{\xi^{(n)}}} & =r^{\sigma_p-1}\Big(a_0+ O(r^{\frac{\alpha_0\sigma_p}{2}})\Big), \\
  |u_{{\bar\xi}^{(i)}{\xi^{(n)}}}| & \lesssim r^{\sigma_p-1}, \ \ i=1,\cdots, n-1,
  \end{split}}
\end{equation}
 where $\bar\xi^{(n)} := \xi^{(n)}$, $\{\bar\xi^{(i)}\}_{i=1}^n$ is an orthonormal basis of $\R^n$ 
and $$a_0 := \lim_{\rho\rightarrow 0^+}\frac{u_{\rho\rho}\big(\frac{\rho x_0}{|x_0|},t_0 \big)}{\rho^{\sigma_p-1}}.$$
Then, by Lemma \ref{lemma-u} and \eqref{u-nn}, we get
\begin{equation}\label{u--nn}
\begin{split}
u^{{\xi^{(n)}}{\xi^{(n)}}}=\frac{U^{{\xi^{(n)}}{\xi^{(n)}}}}{\det D^2 u}& =\frac{U^{{\xi^{(n)}}{\xi^{(n)}}}}{u_{{\xi^{(n)}}{\xi^{(n)}}} U^{{\xi^{(n)}}{\xi^{(n)}}}+O(u_{{\bar\xi}^{(i)}{\xi^{(n)}}} u_{{\bar\xi}^{(j)}{\xi^{(n)}}}  r^{-(n-2)})} \\
&= \frac{1}{u_{{\xi^{(n)}}{\xi^{(n)}}}} \left[1+O\left(\frac{u_{{\bar\xi}^{(i)}{\xi^{(n)}}} u_{{\bar\xi}^{(j)}{\xi^{(n)}}}  r^{-(n-2)}}{u_{{\xi^{(n)}}{\xi^{(n)}}} U^{{\xi^{(n)}}{\xi^{(n)}}}}\right) \right]^{-1} \\
&=\frac{1}{u_{{\xi^{(n)}}{\xi^{(n)}}} } \left[1+O\left(\frac{ r^{2\sigma_p-n}}{u_{{\xi^{(n)}}{\xi^{(n)}}} u^{{\xi^{(n)}}{\xi^{(n)}}} \det D^2u}\right) \right]^{-1}.
\end{split}
\end{equation}
Here we denote by $U^{{\bar\xi^{(i)}}{\bar\xi^{(j)}}}$  the elements of the adjoint matrix of $\{u_{{\bar\xi}^{(i)}{\bar\xi}^{(j)}}\}_{i,j=1}^n$.

Since
$$
{\begin{split}
\det D^2u & \approx (-u_t)^{-\frac1p} \approx r^{-\frac1p}, \\
 \quad u_{{\xi^{(n)}}{\xi^{(n)}}} & \ge \frac{1}{u^{{\xi^{(n)}}{\xi^{(n)}}}},
 \end{split}}
$$
\eqref{u--nn} implies
\begin{equation}\label{u--nn-1}
u^{{\xi^{(n)}}{\xi^{(n)}}}=\frac{1}{u_{{\xi^{(n)}}{\xi^{(n)}}} }\big(1+O(r^{\sigma_p})\big).
\end{equation}

As a result, by \eqref{v-g-1}, \eqref{v-g-2}, \eqref{u-nn} and \eqref{u--nn-1}, we have
\begin{equation}\label{ggrr}
\begin{split}
g g_{{\xi^{(n)}}{\xi^{(n)}}}
&=g\left(\frac{1+\sigma_p}{\sigma_p}v\right)^{-\frac{1}{1+\sigma_p}} v_{{\xi^{(n)}}{\xi^{(n)}}}-\frac{1}{\sigma_p}g_{{\xi^{(n)}}}^2\\
 &= \frac 1{\sigma_p} \left(\frac{1+\sigma_p}{\sigma_p} v\right)^{-\frac{2}{1+\sigma_p}} \Big(\left(\sigma_p + 1\right)(ru_r-u)u^{{\xi^{(n)}}{\xi^{(n)}}}- r^2\Big)\\
  &= \frac{r^2}{\sigma_p} \left(\frac{1+\sigma_p}{\sigma_p} v\right)^{-\frac{2}{1+\sigma_p}} \left(\left(\sigma_p + 1\right)\frac{(ru_r-u)}{r^2u_{{\xi^{(n)}}{\xi^{(n)}}}}(1+O(r^{\sigma_p}))- 1\right)\\
       &= \frac{r^2}{\sigma_p} \left(\frac{1+\sigma_p}{\sigma_p} v\right)^{-\frac{2}{1+\sigma_p}} \left(\frac{ \left(\sigma_p + 1\right) \int_0^r\rho u_{\rho\rho}d\rho}{r^{\sigma_p+1}(a_0+O(r^{\frac{\alpha_0 \sigma_p}{2}}))}(1+O(r^{\sigma_p}))- 1\right)\\
      &\approx  \left( \frac{(\sigma_p +1)\int_0^r\rho u_{\rho\rho}d\rho}{a_0 r^{\sigma_p+1} }\Big(1+O(r^{\frac{\alpha_0 \sigma_p}{2}})\Big) -1 \right) \lesssim   r^{\frac{\alpha_0 \sigma_p}{2}},
\end{split}
\end{equation}
and for $j=1,\cdots, n-1$
\begin{equation}\label{ggrN}
\begin{split}
\left|\sqrt{g}   g_{{\bar\xi}^{(j)}{\xi^{(n)}}}\right|&=\sqrt{g} \left(\frac{1+\sigma_p}{\sigma_p}v\right)^{-\frac{1}{1+\sigma_p}} \left|v_{{\bar\xi}^{(j)}{\xi^{(n)}}}\right| \lesssim r^{\frac{\sigma_p}{2}-1}\left|u^{{\bar\xi}^{(j)}{\xi^{(n)}}} \right| \\
    & \lesssim r^{\frac{\sigma_p}{2}-1} \frac{\left|U^{{\bar\xi}^{(j)}{ \xi^{(n)}}}\right|}{\det D^2 u} \lesssim {\Small\text{$\sum_{i=1}^{n-1}$}} r^{\frac{\sigma_p}{2}-1} \frac{|u_{{\bar\xi}^{(i)}{\xi^{(n)}}}|r^{-(n-2)}}{r^{-1/p}}\\
&\lesssim r^{\frac{\sigma_p}{2}}.
\end{split}
\end{equation}

According to estimates \eqref{h-dydt} and \eqref{HI}, by scaling, we can apply interior estimates for  uniformly parabolic equations \cite{L1996} to get
\begin{equation}\label{htn}
|h_{t,{n+1}}(y',y_{n+1},t)|\lesssim y_{n+1}^{-1}, \quad \forall~(y',y_{n+1},t)\in Q^*_\eta(\bar q_0), \ \ \bar q_0=(\bar y_0',0,\bar t_0).
\end{equation}

By the relationship between $D^2h$ and $D^2g$ in \eqref{g-2tn},  the refined estimates \eqref{ggrr} and \eqref{ggrN} give that
\begin{equation}\label{hnn}
{\begin{split}
      y_{n+1}h_{n+1,n+1} & \lesssim y_{n+1}^{\frac{\alpha_0 }{2}}, \\
   |h_{i,n+1}| & \lesssim 1, \ \ i=1,\cdots,n-1,
   \end{split}} \ \  \text{in} \ \ Q^*_\eta(\bar q_0).
\end{equation}

We claim that \eqref{htn} and \eqref{hnn} imply $h_{n+1}\in C^{\beta}_\mu (\overline{Q^*_\eta(\bar q_0)})$ for some $\beta\in (0,\alpha_0/8)$.  In fact, for all $(y',y_{n+1},t),(\tilde y',\tilde y_{n+1},\tilde t) \in Q^*_\eta(\bar q_0)$, one has
\begin{equation*}
\begin{split}
|h_{n+1}(y',y_{n+1},t)-h_{n+1}(y',\tilde y_{n+1},t)| &\le \left|\int_{\tilde y_{n+1}}^{y_{n+1}}h_{n+1,n+1}(y',\lambda,t)d\lambda\right| \\
&\lesssim |\tilde y_{n+1}-y_{n+1}|^{\frac{\alpha_0}{2}} \lesssim |\sqrt{\tilde y_{n+1}}- \sqrt{y_{n+1}}|^{\frac{\alpha_0}{2}}
\end{split}
\end{equation*}
and
\begin{equation*}
|h_{n+1}(y',y_{n+1},t)-h_{n+1}(\tilde y', y_{n+1},t)|\lesssim |\tilde y'-y'|.
\end{equation*}
Also for $|t-\tilde t|\le y_{n+1}^2$,
\begin{equation*}
|h_{{n+1}}(y',y_{n+1},t)-h_{{n+1}}(y',y_{n+1},\tilde t)|\lesssim |t-\tilde t|y_{n+1}^{-1}\le |t-\tilde t|^{\frac 12};
\end{equation*}
for $|t-\tilde t|\ge y_{n+1}^2$,
\begin{equation*}
\begin{split}
 \ |h_{{n+1}}(y',y_{n+1},t)-h_{{n+1}}(y',y_{n+1},\tilde t)|
& \le |h_{{n+1}}(y',|t-\tilde t|^{\frac 14}, t)-h_{{n+1}}(y',y_{n+1}, t)|\\
 &\hskip10pt +|h_{{n+1}}(y',|t-\tilde t|^{\frac 14},t)-h_{{n+1}}(y',|t-\tilde t|^{\frac 14},\tilde t)|\\
&\hskip10pt  + |h_{{n+1}}(y',|t-\tilde t|^{\frac 14},\tilde t)-h_{{n+1}}(y',y_{n+1},\tilde t)| \\
& \lesssim \big||t-\tilde t|^{\frac 14}-y_{n+1}\big|^{\frac{\alpha_0}2} + |t-\tilde t|^{\frac 34}\\
& \lesssim  |t-\tilde t|^{\frac {\alpha_0}8}.
\end{split}
\end{equation*}
The above estimates conclude that $h_{n+1}\in C^{\beta}_\mu (\overline{Q^*_\eta(\bar q_0)})$.

Moreover, the estimate \eqref{hnn} also gives that
\begin{equation*}
\begin{split}
\hat b (y',0,t)=&-\bigg(\frac{[(n+2)p-1] h_{n+1}}{h_{n+1}^2+y_{n+1}^{\frac{2}{\sigma_p}}(1+|D_{y'}h|^2)}+ \frac{p}{\sigma_p}\widetilde H^{n+1,n+1}\bigg)\bigg|_{y_{n+1}=0}\\
=& -\frac{(n+1)p-1}{h_{n+1}} \gtrsim 1,
\end{split}
\end{equation*}
which yields that the coefficient of $\p_{y_{n+1}}$ in $\mathcal L$
\begin{equation}\label{hat-b-1}
\hat b (y',y_{n+1},t) \gtrsim 1, \quad \forall ~(y',y_{n+1},t) \in Q^*_\eta(\bar q_0).
\end{equation}
By \eqref{h-dydt}, \eqref{HI} and \eqref{hnn}, there holds $$ \hat b/(p\widetilde H^{n+1,n+1}) \gtrsim 1 \quad \text{in}~ Q^*_\eta(\bar q_0).$$
Hence, by Lemma \ref{Lemma-2.3}, we obtain
$h_1,\cdots, h_{n-1}, h_t\in C^{\beta}_\mu (\overline{Q^*_{\frac\eta2}(\bar q_0)})$ for some  $\beta\in(0,1)$.

Consequently,
by the proof of Lemma 4.4  in \cite{HWZ},
or the argument in Section 6 of \cite{DL2004},
we obtain $h\in C^{2+\beta}_\mu (\overline{Q^*_{\frac\eta2}(\bar q_0)})$.
Therefore  the coefficients of the operator $\mathcal L$ belong to $C^{\beta}_\mu (\overline{Q^*_{\frac\eta2}(\bar q_0)})$,
for a small positive constant $\eta$ depending only on $\mathcal M_0, n,p,T$.
By \eqref{gtgx-est}, 
\begin{equation}\label{g-decay}
	g(y,t) \approx \text{dist}(y,\p\Gamma_t)
\end{equation}
near the interface $\Gamma_t$ for $t\in(0,T]$.
Hence $g$ is $C^{2+\beta}_\mu$-smooth up to the interface $\Gamma_{\bar t_0}$, and 
the desired a priori estimate \eqref{g-c2} follows.   
\end{proof}

{\bf \em Proof of Theorem \ref{thm-g}:}

We still consider equation \eqref{h-eq2} in $Q^*_{\eta}(\bar q_0)$ with $\bar q_0=(\bar y_0',0,\bar t_0)$, where $\bar y_0 =(\bar y_0',\bar y_{0,n}) \in \Gamma_{\bar t_0}$, $t_0\in(0,T^*)$.
Differentiating the equation with respect to $t$ gives
\begin{equation*}
\mathcal L (h_t) = 0,
\end{equation*}
where $\mathcal L $ is the linearized operator in \eqref{lin-7}.
Since the coefficients of the operator $\mathcal L$ all belong to $C^{\beta}_{\mu}(\overline{Q^*_{\frac\eta2}(\bar q_0)})$
and $\hat b \gtrsim 1$ in $Q^*_{\frac\eta2}(\bar q_0)$, by Lemma \ref{Lemma-2.4},
one gets $h_t \in C_\mu^{2+\beta}(\overline{Q^*_{\frac\eta2}(\bar q_0)})$.
Similarly, differentiating equation \eqref{h-eq2} in $y_i$, $i=1, \cdots, n-1$,
we have $h_{y_i} \in C_\mu^{2+\beta} (\overline{Q^*_{\frac\eta2}(\bar q_0)})$.
It follows by the Schauder estimate that $D^k_{t,y'} h \in C_\mu^{2+\beta} (\overline{Q^*_{\frac\eta2}(\bar q_0)})$ for each $k \in \mathbb N^+$, after differentiating equation \eqref{h-eq2} with respect to $t, y_i$ up to $k$ times.

As for the regularity of $h$ in $y_{n+1}$, we need to take care of the term $y_{n+1}^{\frac{2}{\sigma_p}}$
in equation \eqref{h-eq2}, which is not smooth if $\frac{2}{\sigma_p}$ is not an integer.

{\bf Case 1:} \
$\frac{2}{\sigma_p}\in \mathbb N^+$.
Then $y_{n+1}^{\frac{2}{\sigma_p}}$ is smooth.
In this case, one can differentiate equation \eqref{h-eq2} in $y_{n+1}$ to obtain higher regularity as above.

{\bf Case 2a:} \
$\frac{2}{\sigma_p}\not\in \mathbb Z^+$ and $\frac{2}{\sigma_p} < 1$.
Let $$z_1= y_1, \ \cdots, \ z_{n-1} = y_{n-1}, \ z_n = 2 \sqrt{y_{n+1}} ,$$
then $h(z,t), h_{z_i}(z,t),h_t(z,t) \in C^{2+\beta}(\overline{Q_\eta(\bar q_0)})$, $i=1,\cdots,n-1$.
Here $Q_\eta(\bar q_0)$, $\bar q_0 = (\bar y'_0, 0, \bar t_0)$, is the cylinder given by
\begin{equation}
Q_\eta(\bar q_0)=\{(z,t)\in \mathbb R^{n,+} \times \mathbb R \ |\ z_n >  0, |z- (\bar y'_0, 0)| < \eta, \bar t_0-\eta^2 < t \le \bar t_0\} .
\end{equation}

Hence, one can rewrite equation \eqref{h-eq2} in coordinates $(z,t)$ as
\begin{equation}\label{hnn-f}
h_{z_nz_n} -  \frac{\sigma_p+2}{\sigma_p}\frac{h_{z_n}}{z_n}
      = \tilde f(z_n^{\frac{4}{\sigma_p}}, h_t, D_{z'}h, \frac{h_{z_n}}{z_n}, D_{z'}D_zh),
\end{equation}
where $\tilde f$ is $C^{1+\beta}$ smooth in its all arguments.
Moreover, by \eqref{hat-b-1},
$$
- \frac{\p \tilde f}{\p (\frac{h_{z_n}}{z_n})} - \frac{\sigma_p+2}{\sigma_p} \gtrsim 1, \quad \text{in} \ \ \overline{Q_\eta(\bar q_0)}.
$$
To prove the regularity of $h_{z_nz_n}$ in $z_n$, we take a fixed point $(z_0,t_0)=(z'_0, 0, t_0) \in \overline{Q_\eta(\bar q_0)}$, and denote
$$\varrho(z_n) = (\varrho_1(z_n),  \varrho_2(z_n))= \Big( z_n^{\frac{4}{\sigma_p}}, \ \frac{h_{z_n}}{z_n} \big|_{(z_0',z_n,t_0)} \Big) \quad\text{with} \ \ (z_0',z_n,t_0)\in \overline{Q_\eta(\bar q_0)}$$
and
 $$\bar f (\varrho(z_n)) =\bar f (\varrho_1(z_n),\varrho_2(z_n)) =  \tilde f(z_n^{\frac{4}{\sigma_p}}, h_t, D_{z'}h, \frac{h_{z_n}}{z_n}, D_{z'}D_zh) \big|_{(z_0',z_n,t_0)}.$$
We also define two constants:
$$\kappa_0 = \frac{h_{z_n}}{z_n} {\big|}_{(z_0,t_0)}, \quad b_0 = - \frac{\p \bar f}{\p \varrho_2} \big|_{\varrho=(0,\kappa_0)} - \frac{\sigma_p+2}{\sigma_p} >0,$$
which are well-defined as $h(z,t) \in C^{2+\beta}(\overline{Q_\eta(\bar q_0)})$.
Then, equation \eqref{hnn-f} can be regarded
as an ODE of the variable $z_n$ and  rewritten as
\begin{equation}\label{hzz-hf}
h_{z_nz_n} \big|_{(z_0',z_n,t_0)}  + b_0  \frac{h_{z_n}}{z_n} \big|_{(z_0',z_n,t_0)}  = \bar f (\varrho(z_n)) - \bar f_{\varrho_2} (0, \kappa_0) \varrho_2(z_n) =: \check f(\varrho(z_n)),
\end{equation}
which yields that
\begin{equation*}
h_{z_n}\big|_{(z_0',z_n,t_0)}  = z_n^{-b_0} \int_{0}^{z_n} \rho^{b_0} \check f(\varrho(\rho)) d \rho
\end{equation*}
and
\begin{equation*}
\frac{h_{z_n}}{z_n}\big|_{(z_0',\lambda,t_0)}  = \lambda^{-b_0-1} \int_{0}^{\lambda} \rho^{b_0} \check f(\varrho(\rho)) d \rho = \int_{0}^1 \rho^{b_0} \check f(\varrho(\lambda\rho)) d \rho.
\end{equation*}
Note that $\p_{\varrho_2}\check f(\varrho(0))=0$.
Then, for $(z_0', \lambda,t_0),(z_0',\tilde\lambda,t_0) \in \overline{Q_\eta(\bar q_0)}$, we get
\begin{equation}
\begin{split}
& \quad \ \Big|\frac{h_{z_n}}{z_n}\big|_{(z_0',\lambda,t_0)} - \frac{h_{z_n}}{z_n}\big|_{(z_0',\tilde\lambda,t_0)} \Big| \le  \int_{0}^1 \rho^{b_0} | \check f(\varrho(\lambda\rho)) - \check f(\varrho(\tilde\lambda\rho)) | d \rho \\
& \le \int_{0}^1   | \check f(\varrho_1(\lambda\rho), \varrho_2(\lambda\rho))  - \check f(\varrho_1(\tilde\lambda\rho), \varrho_2(\lambda\rho))| d \rho \\
&\quad + \int_{0}^1 \rho^{b_0} | \check f(\varrho_1(\tilde\lambda\rho), \varrho_2(\lambda\rho)) - \check f(\varrho_1(\tilde\lambda\rho), \varrho_2(\tilde\lambda\rho)) | d \rho \\
& \le \|D \tilde f\|_{L^{\infty}} \cdot \big| \lambda^{\frac{4}{\sigma_p}} - \tilde \lambda^{\frac{4}{\sigma_p}} \big| + o_{\eta}(1) \|\tilde f\|_{C^{1,\beta}} \cdot \sup_{\rho\in[0,1]} \Big|\frac{h_{z_n}}{z_n}\big|_{(z_0',\lambda\rho,t_0)} - \frac{h_{z_n}}{z_n}\big|_{(z_0',\tilde\lambda\rho,t_0)} \Big|,
\end{split}
\end{equation}
where $o_\eta(1) \le  O\big(\eta^{\beta\min\{1,\frac{4}{\sigma_p}\}} \big) \to 0$ as $\eta \to 0$.
Therefore, for $\eta>0$ small, one gets
\begin{equation}
\begin{split}
\Big|\frac{h_{z_n}}{z_n}\big|_{(z_0',\lambda,t_0)} - \frac{h_{z_n}}{z_n}\big|_{(z_0',\tilde\lambda,t_0)} \Big| \lesssim  |\lambda - \tilde\lambda|^{\min\{1,\frac{4}{\sigma_p}\}}.
\end{split}
\end{equation}
This implies $\frac{h_{z_n}}{z_n}$ is $C^{0,\min\{1,\frac{4}{\sigma_p}\}}$-smooth with respect to $z_n$, so is $h_{z_nz_n}$ from \eqref{hnn-f}.
Recall that $y_{n+1}=\frac {1}4 z_n^2$,
we obtain $h(y',y_{n+1},t)\in C_\mu^{2+\min\{1,\frac{4}{\sigma_p}\}}(\overline{Q^*_{\frac\eta4}(\bar q_0)})$.

{\bf Case 2b:} \
$\frac{2}{\sigma_p}\not\in \mathbb Z^+$ and $\frac{2}{\sigma_p}>1$.
Differentiating equation \eqref{h-eq2} in  $y_{n+1}$ up to $k_0 = \Big[\frac{2}{\sigma_p} \Big]$ times, one gets
$$V := D^{k_0}_{y_{n+1}} h(y',y_{n+1},t) \in C_\mu^{2+\beta}\overline{(Q^*_{\frac\eta2}(\bar q_0)})$$
by  Lemma \ref{Lemma-2.4}.
Similarly,   $V_t, V_{y_i} \in C_\mu^{2+\beta} \overline{(Q^*_{\frac\eta2}(\bar q_0)})$, $i=1,\cdots,n-1$.

Let
$$z_1= y_1, \ \cdots, \ z_{n-1} = y_{n-1},\  z_n = 2 \sqrt{y_{n+1}}, $$
then $V, V_{z_i},V_t \in C^{2+\beta}\overline{(Q_{\eta}(\bar q_0)})$ as a function in $(z,t)$ for $i=1,\cdots,n-1$.
Consider the equation for $V$ in coordinates $(z,t)$ as
\begin{equation}\label{Vnn-f}
V_{z_nz_n} -  \frac{\sigma_p+2}{\sigma_p} \frac{V_{z_n}}{z_n} = \hat f(z_n^{\frac{4}{\sigma_p}-2k_0}, V_t, D_{z'}V, \frac{V_{z_n}}{z_n}, D_{z'}D_zV),
\end{equation}
where $\hat f$ is a $C^{1+\beta}$ smooth function of all its arguments.
Hence, one obtains $V(y',y_{n+1},t)\in C_\mu^{2+\min\{1,\frac4{\sigma_p}-2k_0\} }(\overline{Q^*_{\frac\eta4}(\bar q_0)})$
by the same argument as in Case 2a.

From the arbitrariness of $\bar q_0$, we obtain Theorem \ref{thm-g}.
\qed

\vskip10pt

{\bf \em Proof of Corollary \ref{cor-v}:}

Fix a point $\bar p_0=(\bar y_0, \bar t_0)$ on the interface $\Gamma_{\bar t_0}$, $\bar t_0 \in (0,T]$.
By a rotation of the coordinates,
we may assume that the unit outer normal of the flat part $\{v(y,\bar t_0) = 0\}$ at $\bar p_0$ is $e_n=(0, \cdots, 0, 1)$.

If $\frac{1}{\sigma_p} \in \mathbb Z^+$,
by Theorem \ref{thm-g},  $g$ is $C^\infty$-smooth up to the interface $\Gamma_t$ for $0<t<T^*$.
Hence $v=\frac{\sigma_p}{\sigma_p+1} g^{1+\frac{1}{\sigma_p}}$ is also $C^\infty$ smooth.

Next we consider the case $\frac{1}{\sigma_p} \notin \mathbb Z^+$:

\vskip8pt

If  $\frac{1}{\sigma_p}\in(0,\frac12]$,
then $\frac{2}{\sigma_p}\in (0,1]$ and $k_0 =0$.
By Theorem \ref{thm-g}, we have $g \in C_\mu^{2+ \beta_0}(\overline{\{v>0\}})$, where $\beta_0 := \min\{1, \frac{4}{\sigma_p} - 2 k_0 \}$.
Hence $g \in C^{0,1}(\overline{\{v>0\}})$ and $D_yg \in C^{0,\frac{1}{\sigma_p}}(\overline{\{v>0\}})$
as $\frac{1}{\sigma_p} \le \frac12 \beta_0$.
Hence,
\begin{equation}\label{v-1n-g}
D_yv = g^{\frac{1}{\sigma_p}} D_yg \in C^{0,\frac{1}{\sigma_p}}(\overline{\{v>0\}}),
\end{equation}
which yields $v \in C^{1,\frac{1}{\sigma_p}}(\overline{\{v>0\}})$.

\vskip8pt

If  $\frac{1}{\sigma_p}\in(\frac12, 1)$, then $\frac{2}{\sigma_p}\in (1,2)$ and $k_0 =1$.
By Theorem \ref{thm-g}, $g \in C_\mu^{1, 2+ \beta_0}(\overline{\{v>0\}})$, and so $g, D_yg \in C^{0,1}(\overline{\{v>0\}})$.
Hence, \eqref{v-1n-g} gives $v \in C^{1,\frac{1}{\sigma_p}}(\overline{\{v>0\}})$.

\vskip8pt

If  $\frac{1}{\sigma_p}>1$. Denote $l_0 := [\frac{1}{\sigma_p}]$,
then $\frac{2}{\sigma_p} > 2l_0 \ge 2$.
In this case, by Theorem \ref{thm-g},  $g$ is at least of class $C_\mu^{2l_0, 2+ \vep}(\overline{\{v>0\}})$
for some small  $\vep>0$.
Hence $g, D_yg, D^2_yg, \cdots, D^{l_0+1}_y g \in C^{0,1}(\overline{\{v>0\}})$ as $l_0 + 1 \le 2l_0$.
Differentiating  $D_yv = g^{\frac{1}{\sigma_p}} D_yg$  $l_0$ times in space variables $y$,
we get $D^{l_0+1}_{y} v \in C^{0,\frac{1}{\sigma_p}-l_0}(\overline{\{v>0\}})$.
It follows that $v \in C^{1+l_0, \frac{1}{\sigma_p} - l_0}(\overline{\{v>0\}})$.

\vskip8pt

As for the estimate \eqref{v-inter}, 
we fix $\bar t_0\in[\sigma,T]$ and take $(y, \bar t_0), (\tilde y, \bar t_0) \in \overline{\{0<v(\cdot, \bar t_0)<1\}}$ with 
$d_{y,\tilde y}({\bar t_0}) := \min \{\text{dist}(y,\Gamma_{\bar t_0}), \text{dist}(\tilde y,\Gamma_{\bar t_0})\} = \text{dist}(\tilde y,\Gamma_{\bar t_0}).$

\vskip8pt

If $\frac{2}{\sigma_p} \le 1$, then $k_0  =0$.
By Theorem \ref{thm-g} and \eqref{v-g-2}, it holds
\begin{equation}\label{g-v}
g^{1-\frac{1}{\sigma_p}} v_{ij}= g g_{ij} + \frac{1}{\sigma_p} g_i g_j \in C_\mu^{0,\frac{2}{\sigma_p}}(\overline{\{v>0\}}), \quad 1 \le i,j \le n.
\end{equation}
In the case $|y - \tilde y| \ge  \text{dist}(\tilde y, \Gamma_{\bar t_0})$, \eqref{g-decay} and \eqref{g-v} give that
\begin{equation}\label{d-yy-1}
\begin{split}
d_{y,\tilde y}(\bar t_0)^{1+\frac{1}{\sigma_p}} \frac{ |  v_{ij}(y,\bar t_0) -  v_{ij}(\tilde y, \bar t_0) | }{|y-\tilde y|^{\frac{2}{\sigma_p}}}  \le C \text{dist}(\tilde y, \Gamma_{\bar t_0})^{1 - \frac{1}{\sigma_p}} |  v_{ij}(y,\bar t_0) -  v_{ij}(\tilde y, \bar t_0) | \le C.
\end{split}
\end{equation}
In the case $|y - \tilde y| \le   \text{dist}(\tilde y, \Gamma_{\bar t_0})$, by \eqref{gtgx-est}, \eqref{g-decay} and \eqref{g-v},
we have
\begin{equation}\label{d-yy-2}
\begin{split}
& \quad\ d_{y,\tilde y}({\bar t_0})^{1+\frac{1}{\sigma_p}} \frac{ |  v_{ij}(y,{\bar t_0}) -  v_{ij}(\tilde y, {\bar t_0}) | }{|y-\tilde y|^{\frac{2}{\sigma_p}}} \\
& \le C d_{y,\tilde y}({\bar t_0})^{\frac{2}{\sigma_p}}  \frac{ | g( y,{\bar t_0})^{1-\frac{1}{\sigma_p}} v_{ij}(y,{\bar t_0}) -  g( y,{\bar t_0})^{1-\frac{1}{\sigma_p}} v_{ij}(\tilde y, {\bar t_0}) | }{|y-\tilde y|^{\frac{2}{\sigma_p}}} \\
& \le  C d_{y,\tilde y}({\bar t_0})^{\frac{2}{\sigma_p}} \frac{ |  g(y,{\bar t_0})^{1-\frac{1}{\sigma_p}} v_{ij}(y,{\bar t_0}) -  g(\tilde y,{\bar t_0})^{1-\frac{1}{\sigma_p}} v_{ij}(\tilde y, {\bar t_0}) | }{\mu[(y,\bar t_0),(\tilde y,\bar t_0)]^{\frac{2}{\sigma_p}} \cdot |\sqrt{y_n}+\sqrt{\tilde y_n}|^{\frac{2}{\sigma_p}}} \\
& \quad + C d_{y,\tilde y}({\bar t_0})^{ \frac{2}{\sigma_p}}  \frac{ v_{ij}(\tilde y,{\bar t_0}) |  g(y,{\bar t_0})^{1-\frac{1}{\sigma_p}} -  g(\tilde y,{\bar t_0})^{1-\frac{1}{\sigma_p}} | }{|y-\tilde y|^{\frac{2}{\sigma_p}}} \\
& \le C +  C d_{y,\tilde y}({\bar t_0})^{ \frac{2}{\sigma_p}} v_{ij}(\tilde y,{\bar t_0})\cdot (g_ng^{- \frac{1}{\sigma_p}})\big|_{(\lambda y + (1-\lambda)\tilde y, {\bar t_0})} \cdot |y-\tilde y|^{1- \frac{2}{\sigma_p}} \text{\quad\quad $(\lambda\in[0,1])$}\\
& \le C + C \text{dist}(\tilde y, \Gamma_{\bar t_0})^{ 1- \frac{1}{\sigma_p}} v_{ij}(\tilde y,{\bar t_0}) \le C.
\end{split}
\end{equation}
Hence, \eqref{v-inter} holds when $\frac{2}{\sigma_p} \le 1$.

If $\frac{2}{\sigma_p} >1$, then $k_0 \ge 1$.
By differentiating equation $D_yv = g^{\frac{1}{\sigma_p}} D_y g$ with respect to $y$ $k_0+1$ times and by Theorem \ref{thm-g},
we have
$$g^{k_0+1 - \frac{1}{\sigma_p}} D^{k_0+2}_{y}v \in C_\mu^{0,\frac{2}{\sigma_p}-k_0}(\overline{\{v>0\}}).$$
Therefore, by similar computations as in \eqref{d-yy-1} and \eqref{d-yy-2}, it follows that
\begin{equation*}
d_{y,\tilde y}(\bar t_0)^{1+\frac{1}{\sigma_p}} \frac{ |  D^{k_0+2}_{y}v(y,\bar t_0) -  D^{k_0+2}_{y}v(\tilde y, \bar t_0) | }{|y-\tilde y|^{\frac{2}{\sigma_p} - k_0}} \le C.
\end{equation*}
As a result, \eqref{v-inter} follows.
\qed


\begin{thebibliography}{99}

%
%

\bibitem{A1999} B. Andrews,
            Gauss curvature flow: the fate of the rolling stones.
            {\em Invent. Math.} 138 (1999), no. 1, 151--161.

\bibitem{A2000} B. Andrews,
            Motion of hypersurfaces by Gauss curvature.
            {\em Pacific J. Math.} 195(1) (2000), 1--34.


\bibitem{ACGL}  B. Andrews, B. Chow,  C.  Guenther, M. Langford,
                          Extrinsic Geometric Flows,
                          Graduate Studies in Mathematics 206,
                          Amer. Math. Soc., 2020, xxviii+759 pp.


\bibitem{BCD2017} S. Brendle,  K. Choi and P. Daskalopoulos,
            Asymptotic behavior of flows by powers of the Gaussian curvature.
            {\em Acta Math.}  219 (2017), 1--16.

\bibitem{CEI1999} D. Chopp, L. C. Evans and H. Ishii,
            Waiting time effects for Gauss curvature flows.
            {\em  Indiana Univ. Math. J.} 48 (1999), no. 1, 311--334.

\bibitem{C1985} B. Chow,
             Deforming convex hypersurfaces by the   $n$-th root of the Gaussian curvature.
             {\em J. Differ. Geom.} 22 (1985),  117--138.

\bibitem{D2014} P. Daskalopoulos,
             The regularity of solutions in degenerate geometric problems.
             {\em Surveys in Differential Geometry.} 19 (2014), 83--110.

\bibitem{DH1999} P. Daskalopoulos and R. Hamilton,
              The free boundary in the Gauss curvature flow with flat sides.
              {\em J. Reine Angew. Math.} 510 (1999), 187--227.

\bibitem{DL2003} P. Daskalopoulos and K.-A. Lee,
              H\"older regularity of solutions of degenerate elliptic and parabolic equations.
              {\em J. Funct. Anal.} 201 (2) (2003), 341--379.

\bibitem{DL2004} P. Daskalopoulos and K.-A. Lee,
              Worn stones with flat sides all time regularity of the interface.
              {\em Invent. Math.} 156 (2004), 445--493.

\bibitem{DS2009} P. Daskalopoulos and  O. Savin,
              $C^{1,\alpha}$ regularity of solutions to parabolic Monge-Amp\`ere equations.
              {\em Amer. J. Math.} 134(4) (2014), 1051--1087.



%
%


\bibitem{H1993} R. Hamilton,
              Worn stones with flat sides, in A tribute to Ilya Bakelman.
              {\em Discourses Math. Appl.} 3 (1994), 69--78.


\bibitem{HWZ}
G. Huang, X.-J. Wang and Y. Zhou,
Long time regularity of the $p$-Gauss curvature flow with flat side, preprint, \url{https://arxiv.org/abs/2403.12292}.

\bibitem{HTW} G. Huang, L. Tang and X.-J. Wang,
             Regularity of free boundary for the Monge-Amp\`ere obstacle problem.
              {\em Duke Math J.}, to appear.

\bibitem{JW2014} H. Jian and X.-J. Wang, Optimal boundary regularity for nonlinear singular elliptic equations.
{\em Adv. Math.} 251 (2014), 111--126.


\bibitem{KLR2013} L. Kim, K. Lee and F. Rhee,
             $\alpha$-Gauss curvature flow with flat sides.
             {\em J. Diff. Equa.} 254(3) (2013), 1172--1192.



\bibitem{L1996} G. M. Lieberman,
             Second order parabolic differential equations. {\em World Scientific}, Singapore, 1996.

\bibitem{L2016} G. M. Lieberman,
             Schauder estimates for singular parabolic and elliptic equations of Keldysh type.
             {\em Discrete Contin. Dyn. Syst. Ser. B} 21 (2016), 1525--1566.

\bibitem{T1985} K. Tso,
             Deforming a hypersurface with prescribed Gauss-Kronecker curvature.
             {\em Comm. Pure Appl. Math.} 38 (1985), 867--882.

%
%
%



\end{thebibliography}
\end{document}